\documentclass[a4paper,11pt]{article}
\usepackage[plainpages=false]{hyperref}
\usepackage{amsfonts,latexsym,rawfonts,amsmath,amssymb,amsthm,mathrsfs}
\usepackage{amsmath,amssymb,amsfonts,latexsym,lscape,rawfonts}
\usepackage[usenames]{color}

\usepackage[all]{xy}
\usepackage{eufrak}
\usepackage{makeidx}         
\usepackage{graphicx,psfrag}

\usepackage{array,tabularx}

\usepackage{setspace}

\title{Toric Surfaces, $K$-Stability and Calabi Flow}
\author{Hongnian Huang\thanks{The research of the author is financially supported by FMJH (Fondation Math\'ematique Jacques Hadamard).} }
\date{April 02, 2012}
\usepackage{amsmath,amsthm}
\usepackage{amsmath,amssymb,amsfonts,amsthm,graphics}
\newtheorem{thm}{Theorem}[section]
\newtheorem{cor}[thm]{Corollary}
\newtheorem{lem}[thm]{Lemma}
\newtheorem{prop}[thm]{Proposition}

\newtheorem{conj}[thm]{Conjecture}
\newtheorem{rmk}[thm]{Remark}

\theoremstyle{definition}
\newtheorem{defn}{Definition}[section]

\theoremstyle{remark}

\makeatother
\begin{document}
\maketitle
\renewcommand{\sectionmark}[1]{}

\begin{abstract}
Let $X$ be a toric surface  and $u$ be a normalized symplectic potential on the corresponding polygon $P$. Suppose that the Riemannian curvature is bounded by a constant $C_1$ and 
$
\int_{\partial P} u ~ d \sigma < C_2,
$
then there exists a constant $C_3$ depending only on $C_1, C_2$ and $P$ such that the diameter of $X$ is bounded by $C_3$. Moreoever, we can  show that there is a constant $M > 0$ depending only on $C_1, C_2$ and $P$ such that Donaldson's $M$-condition holds for $u$. As an application, we show that if $(X,P)$ is (analytic) relative $K$-stable, then the modified Calabi flow converges to an extremal metric exponentially fast by assuming that the Calabi flow exists for all time and the Riemannian curvature is uniformly bounded along the Calabi flow.
\end{abstract}

\section{Introduction}
Let $X$ be a polarized K\"ahler manifold with an ample line bundle $L$. The Yau \cite{Y1}, Tian \cite{Ti1} and Donaldson \cite{D1} conjecture says that the existence of cscK metrics is equivalent to that $(X, L)$ is $K$-stable. One way to understand this conjecture is through the geometrical flow. For example, Tosatti \cite{T1} shows that if the curvature along the K\"ahler Ricci flow is uniformly bounded, $(X, L)$ is $K$-stable and asymptotically Chow-semistable, then the K\"ahler Ricci flow converges to a K\"ahler-Einstein metric exponentially fast. 

Geometrical flow not only can help us understand the Yau-Tian-Donaldson conjecture, but it can also help us understand other geometrical phenomenas. For instance, Donaldson \cite{D7} conjectures that the infimum of the Calabi energy is equal to the supremum of all normalized Futaki invaraints of all destabilizing test configurations. Assuming the long time existence of the Calabi flow, Sz\'ekelyhidi \cite{S1} proves Donaldson's conjecture when $X$ is a polarized toric variety. 

It is also expected that the Calabi flow can be related to the Yau-Tian-Donaldson conjecture. Chen \cite{ChenHe3} conjectures that the Calabi flow exists for all time and Donaldson \cite{D6} conjectures that if the Calabi flow exists for all time, then it converges to a cscK metric.

Chen's conjecture is proved in \cite{FH} for the case $X = \mathbb{C}^2 / (\mathbb{Z}^2 + i \mathbb{Z}^2)$ \footnote{In fact, the initial metric needs to be invariant under the translation in the imaginary part of the complex variables. }. Moreover, the curvature is uniformly bounded along the Calabi flow. In the same paper,  the authors also prove Donaldson's conjecture when $X = \mathbb{C}^n / (\mathbb{Z}^n + i \mathbb{Z}^n)$. Thus we should modify Chen and Donaldson's conjectures as follows:

\begin{conj}
Let $[\omega]$ be a K\"ahler class of $X$. Suppose that there is an extremal metric $\omega_0 \in [\omega]$. Then for any K\"ahler metric $\omega_1 \in [\omega]$, the Calabi flow starting from $\omega_1$ exists for all time and the Riemannian curvature is uniformly bounded along the flow.
\end{conj}

\begin{conj}
\label{c1}
Let $[\omega]$ be a K\"ahler class of $X$. Suppose that $\omega_0 \in [\omega]$ is an extremal metric. Let $\omega_1 \in [\omega]$ be any K\"ahler metric invariant under the maximal compact subgroup of the identity component of the reduced automorphism group. If the Calabi flow starting from $\omega_1$ exists for all time and the Riemannian curvature is uniformly bounded along the flow, then the modified Calabi flow converges exponentially fast to an extremal metric in $[\omega]$.
\end{conj}

We want to link Chen and Donaldson's conjectures to the Yau-Tian-Donaldson conjecture. Due to recent developments: \cite{ACGT}, \cite{D5} and \cite{S2}, we only consider the case when $X$ is a toric variety. We have the following conjectures:

\begin{conj}
Suppose $(X,L)$ is relative $K$-stable, then for any K\"ahler metric $\omega \in c_1(L)$, the Calabi flow starting from $\omega$ exists for all time and the Riemannian curvature is uniformly bounded along the flow.
\end{conj}

\begin{conj}
\label{conj}
Suppose $(X,L)$ is relative $K$-stable. Let $\omega \in c_1(L)$ be any K\"ahler metric invariant under the toric action. If the Calabi flow starting from $\omega$ exists for all time and the Riemannian curvature is uniformly bounded along the flow, then the modified Calabi flow converges exponentially fast to an extremal metric in $c_1(L)$.
\end{conj}

In this note, we prove Conjecture (\ref{conj}) when $X$ is a toric surface and $(X,L)$ is (analytic) relative $K$-stable. First, we prove that 

\begin{thm}
\label{dia}
Suppose $u$ is a normalized symplectic potential such that the Riemannian curvature is bounded by $C_1$ and there is a constant $C_2$ such that
$$
\int_{\partial P} u ~ d \sigma < C_2.
$$
Then there is a constant $C(C_1, C_2, P)$  depending only on $C_1, C_2$ and $P$ such that the diameter of $X$ and the maximal value of $u$ is bounded by $C(C_1, C_2, P)$.
\end{thm}

This result enables us to show a compactness result up to diffeomophisms. However, we need to show that the complex structure does not jump when we take limits for a sequence of the Calabi flow. In order to do that, we need to use Donaldson's compactness result in toric surfaces \cite{D3}. The following theorem enables us to use Donaldson's compactness result.

\begin{thm}
\label{M}
Under the assumption of Theorem (\ref{dia}), there exists a constant $M > 0$ depending only on $C_1, C_2, P$ such that $u$ satisfies the $M$-condition.
\end{thm}

Finally we have

\begin{thm}
\label{convergence}
Let $X$ be a toric surface with K\"ahler class $[\omega]$. Let the Delzant polygon of $X$ be $P$. Suppose $(X,P)$ is relative $K$-stable. Let $\omega_1 \in [\omega]$ be any toric invariant K\"ahler metric. If the Calabi flow starting from $\omega_1$ exists for all time and the Riemannian curvature is uniformly bounded along the flow, then the modified Calabi flow converges exponentially fast to an extremal metric in $[\omega]$.
\end{thm}

\begin{rmk}
In \cite{ZZ}, Zhou-Zhu prove that the analytic relative $K$-stability is a necessary condition for the existence of extremal metrics on a toric surface. Thus we also provide a proof of Conjecture (\ref{c1}).
\end{rmk}

{\bf Acknowledgment: } The author is grateful to the consistent support of Professor Xiuxiong Chen, Pengfei Guan, Vestislav Apostolov and Paul Gauduchon. He is also benefited from the discussions of Si Li, G\'abor Sz\'ekelyhidi and Valentino Tosatti. He would like to thank Jeff Streets for his interests in this work.

\section{Notations and Setup}

\subsection{Toric Geometry}
We review the basic theory of toric varieties from Abreu \cite{A1}, Donaldson \cite{D1} and Guillemin \cite{G1} \cite{G2}.

Let $X$ be a $n$-dimensional toric manifold and $L$ be an ample line bundle over $X$. Suppose $\omega \in c_1(L)$ is a toric invariant K\"aher metric. We get a Delzant polytope $P$ through the moment map. The measure $d x$ on the interior of $P$ is the standard Lebesgue measure. The measure $d \sigma$ on the boundary of $P$ is a constant multiplying the standard Lebesgue measure. The constant is $$\frac{1}{|\vec{n}_i|}$$ on each facet $P_i$ of $P$, where $\vec{n}_i \in \mathbb{Z}^n$ is an inward normal vector of $P_i$. The Delzant conditions for $P$ are: for any vertex $v$, there are exactly $n$ facets $P_{i_1}, \ldots, P_{i_n}$ at $v$ and $(\vec{n}_{i_1}, \ldots, \vec{n}_{i_n})$ is a basis of $\mathbb{Z}^n$.

Suppose $P$ has $d$ facets. Let $l_i(x) = \langle x, n_i \rangle - c_i$ and we choose $c_i$ properly such that $l_i = 0$ on $P_i$. Thus $P$ can be expressed as
$$
P =\{x \in \mathbb{R}^n | l_i(x) \geq 0, i = 1, \ldots, d\}.
$$
The Guillemin boundary conditions tell us that the set of symplectic potentials of all toric invariant metrics expressing in $P$ is
$$
\mathcal{H}_{\mathbb{T}} = \{  u_f | f \in C^\infty(\bar{P}),~  u_f = f + \sum_{i=1}^d \frac{1}{2} l_i \ln l_i ~\mathrm{is~ a ~strictly ~ convex~ function ~on ~} P \}.
$$

The Abreu's equation for the scalar curvature $R$ reads

$$
R = - \sum_{i j} u^{ij}_{~ij}.
$$

\subsection{Calabi Flow and Modified Calabi Flow}

By the $\partial\bar{\partial}$-lemma, any toric invariant K\"ahler metric $\omega_\varphi \in c_1(L)$ can be written as
$$
\omega_\varphi = \omega+ \sqrt{-1} \partial \bar{\partial} \varphi,
$$
where $\varphi \in C^\infty(X)$ is a toric invariant function. The Calabi flow equation is
$$
\frac{\partial \varphi}{\partial t} = R_\varphi - \underline{R},
$$
where $R_\varphi$ is the scalar curvature of $\omega_\varphi$ and $\underline{R}$ is the average of $R_\varphi$. By the proof of the short time existence of \cite{ChenHe}, if we start with a toric invariant K\"ahler metric, then the K\"ahler metric along the Calabi flow are all invariant under the toric action.

If we express the Calabi flow equation in $P$, it reads
$$
\frac{\partial f}{\partial t} = \underline{R} - R_f.
$$

Notice that extremal metrics are not the stationary points under the Calabi flow. To understand the behavior of the Calabi flow around extremal metrics, the modified Calabi flow is introduced in \cite{HZ}.

Let us recall the definition of the extremal vector field from \cite{FM}. In general, let $\omega$ be a K\"ahler metric which is invariant under a maximal compact group $G$ of the reduced automorphism group of $X$. Let $\mathcal{J}$ be the holomorphic structure compatible with $\omega$. A function $f$ is called a killing potential if $\mathcal{J}(\nabla f)$ is a killing vector field of the Riemannian metric $g_\omega$, or equivalently, $\nabla f$ is a real holomorphic vector field. A real holomorphic field $\mathcal{X}$ is called the extremal vector field if $\mathcal{J}(\mathcal{X})$ is a killing vector field of the Riemannian metric $g_\omega$ and the potential of $\mathcal{X}$ is the $L^2_\omega$ projection of the scalar curvature $R_\omega$ to the sets of all killing potentials of killing vector fields corresponding to the Lie algebra of $G$. 

Let $\theta$ be the real function on $X$ satisfying 
$$
L_{\mathcal{X}} \omega = \sqrt{-1} \partial \bar{\partial} \theta, \quad \int_X \theta ~ \omega^n = \int_X R_\omega ~ \omega^n.
$$

Let $\varphi$ be a K\"ahler potential invariant under $G$. We define the modified Calabi flow starting from $\omega_{\varphi} = \omega + \sqrt{-1} \partial \bar{\partial} \varphi$ as:
$$
\frac{\partial \varphi}{\partial t} = R_\varphi - \theta_\varphi,
$$
where $\theta_\varphi = \theta + \mathcal{X}(\varphi)$ . 

When $X$ is a toric variety, the modified Calabi flow equation on $P$ reads
$$
\frac{\partial f}{\partial t} = \theta_P - R_f,
$$
where $\theta_P$ is an affine function such that for any affine function $u$
$$
2 \int_{\partial P} u ~ d \sigma - \int_P u \theta_P ~ dx = 0.
$$
Since in the symplectic side, $\theta_P$ is independent of the symplectic potential $u_f$, we will write $\theta$ instead of $\theta_P$ for convenience. 

It is easy to check that when the Futaki invariant vanishes, the modified Calabi flow coincides with the Calabi flow.

Following \cite{D1}, the modified Mabuchi energy is defined as
$$
\mathcal{M}(f) = - \int_P \log \det (D^2 u_f) ~ dx + \mathcal{L}(u_f),
$$
where
$$
\mathcal{L}(u_f) = 2 \int_{\partial P} u_f ~ d \sigma - \int_P u_f \theta ~ d x.
$$

In fact, the modified Calabi flow is the downward gradient flow of the modified Mabuchi energy by the following calculations. Let $\delta f = h$, then
\begin{eqnarray*}
\delta \mathcal{M} (h) &=& - \int_P u_f^{ij} h_{ij} ~ dx + \mathcal{L}(h) \\
&=& - 2 \int_{\partial P} h ~ dx + \int_P R_f h ~ dx + 2 \int_{\partial P} h ~ d \sigma - \int_P h \theta ~ d x\\
&=& \int_P (R_f-\theta) h ~ dx.
\end{eqnarray*}

\subsection{Relative $K$-Stability}

The algebraic relative $K$-stability in \cite{D1} states as follows:

\begin{defn}
$(X, L)$ is algebraic relative $K$-stable if $\mathcal{L} (u) \geq 0$ for all rational piecewise linear function $u$. The equality holds if and only if $u$ is an affine linear function.
\end{defn}

We define the analytic relative $K$-stability as follows:

\begin{defn}
$(X, P)$ is analytic relative $K$-stable if $\mathcal{L} (u) \geq 0$ for all piecewise linear function $u$. The equality holds if and only if $u$ is an affine linear function.
\end{defn}

In this note, we use analytic relative $K$-stability as our relative $K$-stability. 

Let $x_0$ be an interior point of $P$. A convex function $u$ is normalized (at $x_0$) if 
$$
u(x_0) = 0, \quad D u(x_0) = 0.
$$
where $D$ is the Euclidean derivative. Notice that when we normalize a convex function $u$, $\mathcal{M}(u)$ and $\mathcal{L} (u)$ do not change.

Let $\mathcal{C}_\infty$ be the set of continuous convex functions on $\bar{P}$ which are smooth in the interior of $P$. Let $P^*$ be the union of $P$ and its facets and let $\mathcal{C}_1$ be the set of positive convex function $f$ on $P^*$ such that 
$$
\int_{\partial P} f ~ d \sigma < \infty.
$$

Proposition 5.2.2 in \cite{D1} shows that either there is a positive constant $\lambda > 0$ such that 
$$
\mathcal{L} (f) \geq \lambda \int_{\partial P} f ~ d \sigma
$$
for all normalized functions $f$ in $\mathcal{C}_\infty$ or there is a function $f$ in $\mathcal{C}_1$ which is not an affine function and $$\mathcal{L} (f) \leq 0.$$

Let us assume that $(X, P)$ is relative $K$-stable and $X$ is a toric surface. If there is a function $f \in \mathcal{C}_1$ such that $\mathcal{L}(f) < 0$, then Corollary 4.1 in \cite{WZ} tells us that $(X,P)$ is not relative $K$-stable, a contradiction. If for all functions $f \in \mathcal{C}_1, L(f) \geq 0$ and there is a not affine function $f \in \mathcal{C}_1, L(f) = 0$, then Theorem 4.1 in \cite{WZ} shows that $(X,P)$ is not relative $K$-stable. Contradiction again. Thus we conclude the following proposition:

\begin{prop}
\label{stable}
If $(X, P)$ is relative $K$-stable, then there is a constant $\lambda > 0 $ such that
$$
\mathcal{L} (f) \geq \lambda \int_{\partial P} f ~ d \sigma
$$
for all normalized functions $f$ in $\mathcal{C}_\infty$.
\end{prop} 

\section{The Regularity Theorem}

The regularity theorem in Ricci flow is called Shi's estimate \cite{Shi}. Let $M$ be a Riemannian manifold and $g(t), t \in [-1,0]$ be a one parameter Riemannian metric satisfying the Ricci flow equation, i.e.,
$$
\frac{\partial g(t)}{\partial t} = - 2 Ric(t).
$$
Suppose $|Rm(g(t))|$ is bounded by $C_1$ for $t \in [-1,0]$. Then for any integer $k > 0$, there is a constant $C(n, k, C_1)$ depending only on $n, k$ and $C_1$ such that
$$
|\nabla^k Rm(0, x)| < C(n, k, C_1),
$$
for any $x \in M$. The regularity theorem plays an important role in the singularity analysis in Hamilton-Perelman's program, see e.g. \cite{BB}, \cite{CZ}, \cite{KL} and \cite{MT}.

Chen-He \cite{ChenHe2} develop the weak regularity theorem in the singularity analysis of the Calabi flow. Let $X$ be a K\"ahler manifold and $[\omega]$ be a K\"ahler class of $X$. Suppose that the Calabi flow $\omega(t) \in [\omega]$ exists for $t \in [-1, 0]$, the $L^2_{\omega(-1)}$ norm of the bisectional curvature $Rm(-1)$ is bounded by $C_0$ at $t=-1$ and the $L^\infty$ norm of the bisectional curvature $Rm(t)$ is bounded by $C_1$ for $t \in [-1,0]$. Then for any integer $k>0$, there is a constant $C(n,k,C_0, C_1)$ depending only on $n, k, C_0$ and $C_1$ such that
$$
\int_X |\nabla^k Rm(0, x)|^2 ~ \omega_0^2 < C(n,k,C_0,C_1).
$$
Streets also develops similar estimates in \cite{St2}.

In \cite{H1}, the regularity theorem is shown to be one of the obstructions of the long time existence of the Calabi flow on toric varieties. Later, Streets \cite{St} obtains the regularity theorem: suppose the Calabi flow $\omega(t) \in [\omega]$ exists for $t \in [-1, 0]$ and the $L^\infty$ norm of the bisectional curvature $Rm(t)$ is bounded by $C_1$ for $t \in [-1,0]$. Then for any integer $k>0$, there is a constant $C(n,k, C_1)$ depending only on $n, k$ and $C_1$ such that
$$
|\nabla^k Rm(0, x)| < C(n,k, C_1).
$$

Notice that Streets proves the regularity theorem using the sectional curvature and the computations are done in real coordinates. In fact, he needs the following formula for the evolution equation of the sectional curvature:

$$
\frac{\partial Rm}{\partial t} = - \triangle^2 Rm + \nabla^2 Rm * Rm + \nabla Rm * \nabla Rm.
$$

However, the evolution equation of the curvature in \cite{ChenHe2} is written is terms of bisectional curvature and holomorphic coordinates. Thus one needs to redo the calculations. But there is no essential difficulty to go through Streets' calculations.
\section{Diameter Control}

In this section, we prove Theorem (\ref{dia}).

Without loss of generality, we can assume that we work in a standard model, i.e., $O = (0, 0)$ is a vertex of $P$, $x_1, x_2$ are the edges of $P$ around $O$ and $P$ lies in the first quadrant. Guillemin's boundary conditions tell us that 
$$
u= \frac{1}{2} ( x_1 \ln x_1 + x_2 \ln x_2) + f(x_1, x_2),
$$
where $f$ is a smooth function up to the boundary.

Our first observation is the following lemma:

\begin{lem} In edge $x_1$, let $V(x_1) =  \frac{1}{2} x_1 \ln x_1 + f(x_1,0)$, then
$$
\left( \frac{1}{V''} \right)''(x_1) = u^{11}_{~11} (x_1,0).
$$
\end{lem}

\begin{proof}
By direct computations, we have
\begin{eqnarray*}
u^{11}_{~11} & = & \left( \frac{u_{221}}{\det(u_{ij})} -   \frac{u_{22} \det(u_{ij})_1}{(\det(u_{ij}))^2} \right)_1 \\
&=& \frac{u_{2211}}{\det(u_{ij})} -  2 \frac{u_{221} \det(u_{ij})_1}{(\det(u_{ij}))^2} - \frac{u_{22} \det(u_{ij})_{11}}{(\det(u_{ij}))^2 } + 2 \frac{u_{22} (\det(u_{ij})_1)^2}{(\det(u_{ij}))^3}.
\end{eqnarray*}

Let $ v(x_1) = V''(x_1)$. As $x \rightarrow (x_1, 0)$, we get
\begin{eqnarray*}
\lim_{x \rightarrow (x_1, 0)} u^{11}_{~11} (x) & = & \left( - \frac{v''}{v^2} + 2 \frac{v'^2}{v^3} \right) (x_1) \\
& = & \left(- \frac{v'}{v^2} \right)' (x_1) \\
& = & \left( \frac{1}{v} \right) ''(x_1).
\end{eqnarray*}

Hence we obtain the desired result.
\end{proof}

It is shown in \cite{D2} that the norm of Riemannian curvature is expressed as 
$$
|Rm|^2 = \sum u^{ij}_{~kl} u^{kl}_{~ij}.
$$

By direction calculations, we have

\begin{lem} All $u^{ij}_{~kl}(x_1, x_2)$ is finite and
$$
u^{22}_{~11}(x_1, 0) = u^{22}_{~12} (x_1, 0) = u^{12}_{~11} (x_1, 0) = 0
$$
for all $x_1, x_2 \in [0,1]$.
\end{lem}

\begin{proof}
Expressing $u^{ij}_{~kl}(x_1, x_2)$ out, we can see that it is finite. For any $x_1 \in [0,1]$, we have $u^{22}(x_1, 0) = u^{12}(x_1, 0) = 0$. Thus $u^{22}_{~11}(x_1, 0) = u^{12}_{~11}(x_1, 0) = 0$. Also $u^{22}_{~2} (x_1, 0) = 2$ implies $u^{22}_{~12}(x_1,0) = 0$.
\end{proof}

As a corollary, we can simplify the expression of $|Rm|$ in $(x_1, 0)$.

\begin{cor}
$$
|Rm|^2(x_1,0) = ( u^{11}_{~11})^2(x_1,0) + 4 ( u^{12}_{~12})^2(x_1,0) + ( u^{22}_{~22})^2(x_1,0)
$$
\end{cor}

\begin{proof}
Notice that
$$
|Rm|^2(x_1,0) = \sum_{i,j,k,l} u^{ij}_{~kl} u^{kl}_{~ij} (x_1,0).
$$
For each $u^{ij}_{~kl}$, if there are three $1$ or $2$ in $i,j,k,l$, then by the above lemma, we have $u^{ij}_{~kl} u^{kl}_{~ij}(x_1,0) = 0$. If there are exactly two $1$ in $i,j,k,l$ and $i=j$, then we also have $u^{ij}_{~kl} u^{kl}_{~ij} (x_1,0) = 0$. Thus we obtain the conclusion.
\end{proof}

The integral bound of $u$ on the boundary shows that: for every $\epsilon > 0$, there exists a constant $C > 0$ depending only on $\epsilon, C_2$ and $P$ such that for every point $x = (x_1, 0),~ \epsilon \leq x_1 \leq 1$, we have
$$
u(x) < C, \quad  \left|\frac{\partial u}{\partial x_1}\right|(x) < C.
$$

Together with the above lemma, we have

\begin{prop}
\label{grad}
There exists a constant $C_3 > 0$ depending only on $ C_1, C_2$ and $P$ such that
$$
u(0,0) < C_3, \quad |\nabla f|(0,0) < C_3.
$$
\end{prop}

\begin{proof}
Let $V(x_1) = \frac{1}{2} x_1 \ln x_1 + f(x_1, 0)$. Without loss of generality, we only need to show that there exists a constant $C_3 > 0$ depending only on $C_1, C_2$ and $P$ such that
$$
V(0) < C_3,\quad \left| \frac{\partial f}{\partial x_1} \right| (0) < C_3.
$$

It is easy to see that 
$$
\left( \frac{1}{V''} \right)' (0) = 2.
$$

Let us pick $\epsilon = \frac{1}{C_1}$, then for any $s \in (0, \epsilon]$, we have
$$
\left| \left( \frac{1}{V''} \right)' (s) - 2 \right| = \left| \int_0^s \left( \frac{1}{V''} \right)'' (x) \ dx \right| \leq C_1 s.
$$
Hence
$$
2 - C_1 s \leq  \left( \frac{1}{V''} \right)' (s) \leq  2+ C_1 s.
$$
Since 
$$
\frac{1}{V''} (0) = 0,
$$
we have
$$
2s - \frac{C_1}{2}s^2 \leq \frac{1}{V''}(s) \leq 2s + \frac{C_1}{2}s^2.
$$

In terms of $f(x_1,0)$, we have
$$
-\frac{C_1}{8+2 C_1 s} \leq \frac{\partial^2 f}{\partial x_1^2} (s, 0)  \leq \frac{C_1}{8-2 C_1 s}.
$$

Since $|\frac{\partial f}{\partial x_1}(\epsilon,0)| < C$, we conclude that $|\frac{\partial f}{\partial x_1}(s,0)| < C$ for all $s \in [0, \epsilon]$. It is also easy to control the $C^0$ norm of $f$ at $(0,0)$.
\end{proof}

Next we want to show that the diameter of $(X, u)$ is bounded by a constant $C$ depending only on $C_1, C_2$ and $P$. Without loss of generality, we only need to show that for any point $ x = (x_1, x_2)$ with $x_1, x_2 \geq 0, ~ x_1 + x_2 \leq 1$, the Riemannian distance $d_u (O, x) < C$. 

Let the vector $\vec{v} = \langle a, b \rangle, ~ a, b \geq 0, ~ a + b =1$ be a vector pointing from $O$ to $x$. We can parametrize the line interval from $O$ to $x$ as following:

$$
x_1(t) = a t, \quad x_2(t) = b t, \quad t \in [0,1].
$$

Let $V(t)$ be the restriction of $u$ on the above line interval. Then
$$
V(t) = \frac{1}{2} ( at \ln (at)  + bt \ln (bt) ) + f(at, bt) = \frac{1}{2} t \ln t + g(t),
$$ 
where $g(t) = \frac{1}{2} (at \ln a + bt \ln b) + f(at, bt)$.

Notice that since  $$|\nabla f \cdot \vec{v}|(0) < C,$$ we have $|g'(0)| < C$. In fact, we can prove the following lemma.

\begin{lem}
There is a constant $C > 0$ depending only on $C_1, C_2$ and $P$ such that for every $t \in [0,1]$, 
$$
|g'(t)| < C.
$$
\end{lem}

\begin{proof}
Pick a small constant $\epsilon > 0$, we only need to control $g'(t)$ for $t \in [0, \epsilon]$. Since $V(t)$ is a convex function and $V(0), V(1) \geq 0, ~ V(0), V(1) \leq C(C_1, C_2, P)$, we can control $g'(\epsilon)$ by a constant depending only on $\epsilon, C_1, C_2$ and $P$.

By Lemma 3 in \cite{D3}, we have
$$
\left( \frac{1}{V''} \right) '' (t) \leq |Rm|(x(t), y(t)).
$$

The arguments in Proposition (\ref{grad}) show that 

$$
g''(t) > -\frac{C_1}{8+2 C_1 t}.
$$

Together with the fact $g'(0), g'(\epsilon)$ are bounded, we conclude that $g'(t)$ is bounded by a constant depending only on $C_1, C_2$ and $P$.

\end{proof}

As a result, we have
 
\begin{cor}
\label{dis}
There exists a constant $C > 0$ depending only on $C_1, C_2$ and $P$ such that
$$d_u (O, x) < C(C_1, C_2, P).$$
\end{cor}

\begin{proof}
We split the interval $ [0,1]$ into infinite many intervals as $$\left[\frac{1}{2}, 1\right],\left[\frac{1}{4}, \frac{1}{2}\right], \ldots, \left[\frac{1}{2^n}, \frac{1}{2^{n-1}}\right], \ldots$$. 
Then
\begin{eqnarray*}
d_u(O,x) &\leq& \int_0^1 \sqrt{V''(t)} ~ dt \\
&=& \lim_{N \rightarrow \infty} \sum_{i=0}^N \int_{\frac{1}{2^{i+1}}}^{\frac{1}{2^i}} \sqrt{V''(t)} ~ dt \\
&\leq&  \lim_{N \rightarrow \infty} \sum_{i=0}^N \sqrt{ \frac{1}{2^{i+1}}\int_{\frac{1}{2^{i+1}}}^{\frac{1}{2^i}} V''(t) ~dt} \\
&=& \lim_{N \rightarrow \infty} \sum_{i=0}^N \sqrt{ \frac{1}{2^{i+1}} \left(\frac{\ln 2}{2} + g'\left(\frac{1}{2^{i}}\right) - g'\left(\frac{1}{2^{i+1}}\right ) \right) ~dt} \\
&\leq& C(C_1, C_2, P) \sum_{i=0}^\infty \frac{1}{(\sqrt{2})^{i+1}} \\
&\leq& C(C_1, C_2, P) 
\end{eqnarray*}

\end{proof}

\begin{proof}[Proof of Theorem (\ref{dia})] By Proposition (\ref{grad}) and Corollary (\ref{dis}), we obtain the result.

\end{proof}

\section{$M$-condition}

We recall the definition of $M$-condition from \cite{D3}. Let $\overline{x_1x_4}$ be any segment of $P$ and $x_2, x_3$ be two points in $\overline{x_1x_4}$ such that $|\overline{x_1x_2}| = |\overline{x_2x_3}| = |\overline{x_3x_4}|$. Let $\vec{v}$ be an unit vector pointing from $x_1$ to $x_4$. $u$ satisfies  the $M$-condition on $\overline{x_1x_4}$ if 
$$
\left| D u \cdot \vec{v}(x_2) - D u \cdot \vec{v}(x_3) \frac{}{}\right| < M.
$$

\begin{defn}
$u$ satisfies the $M$-condition on $P$ if for any segment $l \subset P$, $u$ satisfies the $M$-condition on $l$.
\end{defn}

Let us write $u(x_1, x_2) = \frac{1}{2} (x_1 \ln x_1 + x_2 \ln x_2) + f(x_1, x_2)$. Proposition (\ref{grad}) shows that 
$$
\left| \frac{\partial f}{\partial x_1}(x_1, 0) \right|,~ \left| \frac{\partial f}{\partial x_2}(0, x_2) \right| < C(C_1, C_2, P),
$$
for any $x_1, x_2 \in [0,1]$. Our next few lemmas show that we can also control $| \frac{\partial f}{\partial x_1}(0, x_2) |$.

\begin{lem}
$$
u^{12}_{~12}(0, x_2) = - \frac{\partial \frac{2 f_{12}}{f_{22} + \frac{1}{2 x_2}}}{\partial x_2} (0, x_2),
$$
where $x_2 \in (0,1]$.
\end{lem}

\begin{proof}
This is done by direct calculations.

\begin{eqnarray*}
u^{12}_{~12}(0, x_2) & = &- \lim_{x_1 \rightarrow 0} \frac{\partial^2 \frac{f_{12}}{u_{11}u_{22} - f_{12}^2}}{\partial x_1 \partial x_2}(x_1, x_2)\\
& = &- \lim_{x_1 \rightarrow 0} \frac{\partial \frac{2 f_{12}}{f_{22} + \frac{1}{2 x_2}}}{\partial x_2} (x_1, x_2)\\
& = & - \frac{\partial \frac{2 f_{12}}{f_{22} + \frac{1}{2 x_2}}}{\partial x_2} (0, x_2)
\end{eqnarray*}
\end{proof}

\begin{prop}
There exists a constant $C$ depending only on $C_1, C_2$ and $P$ such that for any $x_2 \in (0,1]$, we have
$$
\left|\frac{\partial f}{\partial x_1} (0, x_2) \right| < C.
$$
\end{prop}

\begin{proof}
Combining the previous results, we have
$$
\left| \frac{\partial \frac{2 f_{12}}{f_{22} + \frac{1}{2 x_2}}}{\partial x_2} (0, x_2) \right| = \left|  u^{12}_{~12}(0, x_2) \right| < C_1.
$$

Notice that  
$$
\lim_{x_2 \rightarrow 0}  \frac{2 f_{12}}{f_{22} + \frac{1}{2 x_2}} (0, x_2) = 0.
$$

Then

\begin{eqnarray*}
\left|  \frac{2 f_{12}}{f_{22} + \frac{1}{2 x_2}} (0, x_2) \right| < C x_2.
\end{eqnarray*}

When $x_2$ is close to $0$, Proposition (\ref{grad}) tells us that $f_{22}(0, x_2)$ is bounded. When $x_2$ is away from $0$, Lemma 4 of \cite{D3} tells us that $f_{22} + \frac{1}{2 x_2}$ is bounded. So we conclude that for $x_2 \in [0,1]$,
$$
f_{12} (0, x_2) < C.
$$
Together with the fact that $f_1(0, 0)$ is bounded, we obtain the conclusion.
\end{proof}

\begin{prop}
For any $x_1, x_2 \in [0,1]$, we have
$$
| \nabla f (x_1, x_2) | < C(C_1, C_2, P).
$$
\end{prop}

\begin{proof}
Without loss of generality, we only need to control $\frac{\partial f}{\partial x_1}(x_1, x_2)$ for any $x_1, x_2 \in (0,1]$. In fact, we only need to consider the case when $x_1$ is small. Let us pick a constant $\epsilon > x_1$. It is easy to see that $\frac{\partial f}{\partial x_1}(\epsilon, x_2)$ is controlled by a constant depending only on $\epsilon, C_1, C_2$ and $P$. Lemma 3 in \cite{D3} shows that 
$$
\left( \frac{1}{u_{11}} \right)_{11} (s, x_2) < C_1
$$
for all $s \in [0, \epsilon]$. Let $V(t) = u(t, x_2), t \in [0, \epsilon]$, we have
$$
\left( \frac{1}{V''} \right)'' (t) < C.
$$
The calculations in Proposition (\ref{grad}) show that 
$$
f_{11} (s, x_2) > C
$$
for all $s \in [0, \epsilon]$ since $\frac{1}{2} x_2 \ln x_2$ has no contribution. Together with the fact that $\frac{\partial f}{\partial x_1}(0, x_2)$ and $\frac{\partial f}{\partial x_1}(\epsilon, x_2)$ are controlled, we obtain that $\frac{\partial f}{\partial x_1}(s, x_2)$ is controlled for all $s \in (0,\epsilon)$.
\end{proof}

As a consequence, we have proved Theorem (\ref{M}).

\section{Compactness}

Let $u^{(\alpha)}(t, x), ~ t \in [-1,0],~ x \in P$ be a sequence of modified Calabi flows  on $P$ satisfying:

\begin{itemize}

\item For any $\alpha$, the Riemannian curvature of $u^{(\alpha)}(t,x), ~ t \in [-1, 0], x \in P$ is bounded by $ C_1$ and $u^{(\alpha)}(0,x)$ is normalized.

\item For any $\alpha$ and any $t \in [-1, 0]$, let $\tilde{u}^{(\alpha)}(t,x)$ be the normalization of $u^{(\alpha)}(t,x)$, then
$$
\int_{\partial P} \tilde{u}^{(\alpha)}(t,x) ~ d \sigma < C_2.
$$
\end{itemize}

By the results of the previous section, we conclude that  $\tilde{u}^{(\alpha)}(t,x)$ satisfies the $M$-condition for all $\alpha$ and $t \in [-1,0]$. Thus the injectivity radius of $(X, u^{(\alpha)}(t,x))$ is bounded from below by Proposition 4 of \cite{D3}. By applying the weak regularity theorem or the regularity theorem \footnote{According to \cite{HZ}, the modified Calabi flow is just the pull back of the Calabi flow by a one parameter group diffeomorphisms generating by the real extremal vector field. We apply the weak regularity theorem or the regularity theorem to the corresponding Calabi flow to get the estimates. Then we get the same estimates for the modified Calabi flow.}, we obtain
\begin{eqnarray}
\label{reg}
|\nabla^k Rm^{(\alpha)}|(t, x) \leq C(k, C_1, C_2, P),
\end{eqnarray}
for all $\alpha,~ x \in P$ and $t \in [-\frac{1}{2},0]$. 

For any $\epsilon > 0$, let $P_\epsilon$ be the region in $P$ whose point is away from $\partial P$ with Euclidean distance at least $\epsilon$. By Lemma 4 and Lemma 6 of \cite{D3}, we can control $\left( \frac{\partial^2 u^{(\alpha)}(t, x)}{\partial x_i \partial x_j} \right)$, for any $\alpha, ~  t \in [-\frac{1}{2},0]$ and $x \in P_\epsilon$. 

Let us temporarily suppress $\alpha$ and consider the Abreu's equation for $t \in [-\frac{1}{2},0]$ and $x \in P_\epsilon$:

$$
- U^{ij} \left( \frac{1}{\det(D^2 u(t,x))}\right)= R(t,x).
$$

By (\ref{reg}) and Corollary 5.2 of \cite{FH}, we can control the $C^\infty$ norm  of $R(t, x)$, where the $C^\infty$ norm is for $x$ coordinates and is measured in terms of the standard Euclidean metric. Thus  we can control the $C^\infty$ norm of $u(t,x)$ by Shauder's estimates. Again, the $C^\infty$ norm is for $x$ coordinates and is measured in terms of the standard Euclidean metric. In order to control the derivatives of $u(t,x)$ in terms of $t$ and the mixed derivatives of $t$ and $x$, one can use the result of Proposition 5.4 of \cite{FH}. Notice that the Euclidean $C^0$ and $C^1$ norm of $u^{(\alpha)}(t, x)$ in space direction are controlled for $t \in [-\frac{1}{2},0], x \in P_\epsilon$ because $R(t,x)$ and $D(R(t,x))$ are controlled, where $D(R(t,x))$ is the Euclidean derivative in the space direction. Our discussions lead to the following result:

\begin{thm} 

\label{con}

By passing to a subsequence, $u^{(\alpha)}(t,x)$ converges to a smooth function $u(t,x),~ t\in [-\frac{1}{2}, 0],~ x \in P$. Moreover $u(t,x)$ are symplectic potentials on $P$ satisfying the modified Calabi flow equation.
\end{thm}

\begin{proof}
The only thing that we need to prove is that $u(t,x)$ satisfies the Guillemin boundary conditions. The proof is exactly as the proof of Proposition 8 of \cite{D3}. 
\end{proof}

\section{Relative $K$-Stability}
By Proposition (\ref{stable}) in the introduction section and Proposition 5.1.2 of \cite{D1}, we know that the modified Mabuchi energy is bounded from below in $\mathcal{C}_\infty$. Moreover, if $u^{(\alpha)}$ is any sequence of normalized functions in  $\mathcal{C}_\infty$ which is a minimizing sequence of the modified Mabuchi energy, Proposition 5.1.8 of \cite{D1} shows that
$$
\int_{\partial P} u^{(\alpha)} \leq C_2.
$$

Now we suppose that the Calabi flow $u(t,x)$ exists for all time and the Riemannian curvature is uniformly bounded by $C_1$. Then the corresponding modified Calabi flow exists for all time and the Riemannian curvature is also uniformly bounded by $C_1$. We still denote the modified Calabi flow as $u(t,x)$. Let us take a sequence of $t_i \rightarrow \infty$ and define a sequence of the modified Calabi flow by
$$
u^{(\alpha)}(t,x) = u(t_\alpha+t, x),~ t \in [-1,0], x \in P.
$$
For each $\alpha$, we add some affine function $l^{(\alpha)}(x)$ to $u^{(\alpha)}(t, x)$ such that  $u^{(\alpha)}(0, x)$ is normalized. Notice that the modified Mabuchi energy does not change if we add $l^{(\alpha)} (x)$ to $u^{(\alpha)}(t, x)$. With the fact that the modified Calabi flow decreases the modified Mabuchi energy, we conclude
$$
\int_{\partial P} \tilde{u}^{(\alpha)}(t,x) ~ dx \leq C_2,
$$
where $ \tilde{u}^{(\alpha)}(t,x) $ is the normalization of $u^{(\alpha)}(t, x)$ for any $\alpha$ and any $t \in [-1,0]$.

Theorem (\ref{con}) tells us that $u^{(\alpha)}(t, x)$ converges to $u^{(\infty)}(t, x),~ t \in [-\frac{1}{2},0], x \in P$ by passing to a subsequence. Since the modified Mabuchi energy is bounded from below, the modified Mabuchi energy of $u^{(\infty)}(0, x)$ is the infimum of the modified Mabuchi energy of $u(t,x)$. We can also see that the modified Mabuchi energy of $u^{(\infty)}(t, x), ~ t \in [-\frac{1}{2},0]$ is the infimum of the modified Mabuchi energy of the original modified Calabi flow. Thus the scalar curvature of $u^{(\infty)}(t,x)$ is $\theta$.

In order to show that $u(t_i, x)$ converges to an extremal metric, we only need to show that $l_\alpha(x)$ is bounded. Calabi and Chen show that the Calabi flow decreases the geodesic distance in \cite{CC}. This result also holds for the modified Calabi flow, i.e.,
$$
\int_P (u(t,x)-u^{\infty}(0,x))^2 ~ dx
$$
is decreasing as $t$ increases. It tells us that
$$
\int_P u^2(t,x) ~ dx
$$
is bounded. Together with the fact that
$$
\int_P (u^{(\alpha)}(0,x))^2 ~ dx
$$
is bounded, we conclude that the affine function $l^{(\alpha)}(x)$ is bounded. Thus by passing to a subsequence, $u(t_i,x)$ converges to an extremal metric $u_\infty(x)$. Let us write the K\"ahler metric of $u(t_i,x), u_\infty(x)$ as $\omega(t_i, x), \omega_\infty(x)$ respectively. Let $\varphi_i$ be the difference of the Legendre transform of $u(t_i,x)$ and $u_\infty(x)$. We obtain
$$
 \omega(t_i) = \omega_\infty + \sqrt{-1} \partial \bar{\partial} \varphi_i.
$$
In page 118 of \cite{D4}, Donaldson shows that
$$
||\varphi_i||_{L^\infty} = ||u(t_i) - u_\infty||_{L^\infty}.
$$
Thus the $L^\infty$ norm of $\varphi_i$ is independent of $i$. Since the $L^\infty$ norm of $Ric(\varphi_i)$ is also independent of $i$, we conclude that the $C^{3,\alpha}$ norm of $\varphi_i$ is also independent of $i$ by Theorem 5.1 of \cite{ChenHe}. Notice that the distance between $u(t_i,x)$ and $u_\infty(x)$ goes to 0 as $i$ goes to infinity. Applying the results of \cite{HZ}, we obtain the exponential convergence of the modified Calabi flow. Hence we complete the proof of Theorem (\ref{convergence}).

\section{Discussion}

It is easy to see that the analytic relative $K$-stability is stronger than the algebraic relative $K$-stability. But when the Futaki invariant vanishes, Donaldson \cite{D1} shows that algebraic $K$-stability implies that there exists a constant $\lambda > 0$ such that for all normalized convex function $f$ in $\mathcal{C}_\infty$, we have
$$
\mathcal{L} (f) > \lambda \int_{\partial P} f ~ d \sigma.
$$
Thus we obtain the following stronger result when the Futaki vanishes.

\begin{thm}
Let $X$ be a toric surface with ample line bundle $L$. Suppose $(X,L)$ is algebraic $K$-stable with vanishing Futaki invariant. Let $\omega \in c_1(L)$ be any toric invariant K\"ahler metric. If the Calabi flow starting from $\omega$ exists for all time and the Riemannian curvature is uniformly bounded along the flow, then the Calabi flow converges exponentially fast to a cscK metric in $c_l(L)$.
\end{thm}

It is an interesting question whether the algebraic $K$-stability implies the analytic $K$-stability when $X$ is a toric surface.

Hongnian Huang, \ hnhuang@gmail.com

CMLS

Ecole Polytechnique

\begin{thebibliography}{10}

\bibitem{A1} M. Abreu,  {\em K\"ahler geometry of toric varieties and extremal metrics}, International
J. Math. 9 (1998), 641-651.

\bibitem{ACGT} V. Apostolov,~ D.M.J. Calderbank, ~P. Gauduchon and C.W. Tonnesen-Friedman, {\em Hamiltonian 2-forms in K\"ahler geometry III: Extremal metrics and stability}, Invent. Math. 173 (2008), 547-601.

\bibitem{BB} L. Bessi\`eres,  G. Besson,  M. Boileau, S. Maillot and  J. Porti, {\em The Geometrisation of 3-manifolds}, EMS Tracts in Mathematics, volume 13, European Mathematical Society, Zurich, 2010.

\bibitem{CC} E. Calabi and X.X. Chen, {\em Space of K\"ahler metrics and Calabi flow}, J.
Differential Geom. 61 (2002), no. 2, 173-193.

\bibitem {CZ} H.D. Cao and X.P. Zhu, {\em A Complete Proof of the Poincar\'e and Geometrization Conjectures - Application of the Hamilton-Perelman theory of the Ricci flow}, Asian J. Math. 10 (2006), no.2, 165-492.

\bibitem{ChenHe} X.X. Chen and W.Y. He, {\em On the Calabi flow}, Amer. J. Math. 130 (2008), no. 2, 539-570.

\bibitem{ChenHe2} X.X. Chen and W.Y. He, {\em The Calabi flow on K\"ahler surface with bounded Sobolev constant--(I)}, arXiv:0710.5159.

\bibitem{ChenHe3} X.X. Chen and W.Y. He {\em The Calabi flow on toric Fano surface}, arXiv:0807.3984.

\bibitem{D1} S.K. Donaldson, {\em Scalar curvature and stability of toric varieties}, Jour. Differential
Geometry 62 (2002), 289-349.

\bibitem{D2} S.K. Donaldson, {\em Interior estimates for solutions of Abreu's equation}, Collectanea
Math. 56 (2005), 103-142.

\bibitem{D3} S.K. Donaldson, {\em Extremal metrics on toric surfaces: a continuity method}, J. Differential Geom. 79 (2008), no. 3, 389-432.

\bibitem{D4} S.K. Donaldson, {\em Constant scalar curvature metrics on toric surfaces}, Geom. Funct. Anal. 19 (2009), no. 1, 83-136. 

\bibitem{D5} S.K. Donaldson, {\em b-Stability and blow-ups},  arXiv:1107.1699.

\bibitem{D6} S.K. Donaldson, {\em Conjectures in K\"ahler geometry}, Strings and geometry, 71-
78, Clay Math. Proc., 3, Amer. Math. Soc., Providence, RI, 2004.

\bibitem{D7} S.K. Donaldson, {\em Lower bounds on the Calabi functional}, J. Differential
Geom., 70(3):453-472, 2005.

\bibitem{FH} R.J. Feng and H.N. Huang, {\em The Global Existence and Convergence of the Calabi Flow on $\mathbb{C}^n = \mathbb{Z}^n + i \mathbb{Z}^n$}, Preprint.

\bibitem {FM} A. Futaki and T. Mabuchi, {\em Bilinear forms and extremal K\"ahler vector fields associated with K\"ahler
classes}, Math. Ann. 301 (1995), 199-210.

\bibitem{G1} V. Guillemin, {\em K\"ahler structures on toric varieties}, J. Differential Geom. 40
(1994), 285-309.

\bibitem{G2} V. Guillemin, {\em Moment maps and combinatorial invariants of Hamiltonian $T^n$-
spaces}, Birkhauser, 1994.

\bibitem{H1} H.N. Huang, {\em On the Extension of the Calabi Flow on Toric Varieties}, to appear in Annals of Global Analysis and Geometry, arxiv:1101.0638.

\bibitem{HZ} H.N. Huang and K. Zheng, {\em Stability of Calabi flow near an extremal metric}, to appear in Annali Della Scuola Normale Superiore Di Pisa, arXiv:1007.4571.

\bibitem{KL} B. Kleiner and J. Lott, {\em Notes on Perelman's papers}, Geometry \& Topology 12 (2008), 2587-2855.

\bibitem{MT} J. Morgan and G. Tian, {\em Ricci flow and the Poincar\'e Conjecture},  Clay Mathematics Monographs, 3. American Mathematical Society, Providence, RI; Clay Mathematics Institute, Cambridge, MA, 2007.

\bibitem{Shi} W.X. Shi, {\em Ricci deformation of the metric on complete noncompact Riemannian manifolds} , J. Differential Geom, 30(2) (1989), 303-394.

\bibitem{St} J. Streets, {\em The long time behavior of fourth-order curvature flows}, to appear in Calc. Var. PDE.

\bibitem{St2} J. Streets, {\em The Gradient Flow of $\int_M |Rm|^2$}, J. Geom. Anal. 18 (2008), no. 1, 249-271.

\bibitem{S1} G. Sz\'ekelyhidi, {\em Optimal test-configurations for toric varieties}, J. Differential Geom. 80 (2008), 501-523.

\bibitem{S2} G. Sz\'ekelyhidi, {\em Filtrations and test-configurations}, arXiv:1111.4986.

\bibitem{Ti1} G. Tian, {\em K\"ahler-Einstein metrics of positive scalar curvature}, Inventiones
Math. 130 1-57 (1997)

\bibitem{T1} V. Tosatti, {\em K\"ahler-Ricci flow on stable Fano manifolds}, J. Reine Angew. Math. 640 (2010), 67-84. 

\bibitem{WZ} X. Wang and B. Zhou, {\em On the existence and nonexistence of extremal metrics on toric K\"ahler surfaces}, Adv. in Math. 226 (2011), 4429-4455.

\bibitem{Y1} S.T. Yau, {\em Review of K\"ahler-Einstein metrics in algebraic geometry}, Israel
Math. Conference Proc., Bar-Ilan Univ. 9 433-443 (1996)

\bibitem{ZZ} B. Zhou, X.H. Zhu, {\em $K$-stability on toric manifolds}, Proc. Amer. Math. Soc. 136 (2008), no. 9, 3301-3307.

\end{thebibliography}
\end{document}